\documentclass[12pt,leqno]{amsart}
\usepackage{amsmath}
\usepackage{amsthm}
\usepackage{amssymb}
\newtheorem{corollary}[equation]{Corollary}

\newtheorem{lemma}[equation]{Lemma}

\newtheorem{proposition}[equation]{Proposition}

\newtheorem{theorem}[equation]{Theorem}

\newcommand{\C}{{\mathbb{C}}}
\newcommand{\A}{{\mathbb{A}}}

\renewcommand{\P}{{\mathbb{P}}}

\newcommand{\Z}{\mathbf{Z}}

\numberwithin{equation}{section}

\title[Generalization of uniqueness theorem]{Generalization of uniqueness theorem for meromorphic mappings sharing hyperplanes} 

\date { }
\author{Si Duc Quang}

\begin{document}

\maketitle


\renewcommand{\thefootnote}{\empty}



\renewcommand{\thefootnote}{\arabic{footnote}}
\setcounter{footnote}{0}


\begin{abstract} {In this article, by introducing a new method in estimating the counting function of the auxiliary function, we prove a new generalization of uniqueness theorems for meromorphic mappings into $\P^n(\C )$ which share few hyperplanes regardless of multiplicities. Our result improves the previous result in this topic. Moreover our proof is simpler than the previous proofs.}
\end{abstract}

\section{Introduction}
Let $f : \C^m \longrightarrow \P^n(\C)$ be a meromorphic mapping with a reduced representation
$f = (f_0 : \dots : f_n)$.
Let $H$ be a hyperplane in $\P^n(\C)$ given by $H=\{a_0\omega_0+\cdots+a_n\omega_n=0\}.$ 
We set  $(f,H)=\sum_{i=0}^na_if_i$. This function depends on the choices of presentations of $f$ and $H$. However its divisor $\nu_{(f,H)}$ is independent from these choices.

Consider $f$ and $g$ be two linearly non-degenerate meromorphic mappings from $\C^m$ into $\P^n(\C)$. Let $\{H_i\}_{i=1}^q$ be $q$ hyperplanes of $\P^n(\C)$ located in general position, satisfying 
$$ \dim (f,H_i)^{-1}\{0\}\cap (f,H_j)^{-1}\{0\}\le m-2\ \forall 1\le i<j\le q.\ \ \ (*)$$
Assume that

(i) $\min\{\nu_{(f,H_i)},d\}=\min\{\nu_{(g,H_i)},d\}$ for all $1\le i\le q$,

(ii) $f=g$ on $\bigcup_{i=1}^q(f,H_i)^{-1}\{0\}.$

\noindent
Here, a family of hyperplanes in $\P^n(\C)$ is said to be \textit{in general position} if the intersections of any $n+1$ hyperplanes among them are empty.  
We will say that such two meromorphic mappings $f$ and $g$ above share $q$ hyperplanes $H_i\ (1\le i\le q)$ with multiplicities counted to level $d$. If $d=+\infty$ (resp. $d=1$), $f$ and $g$ are said to share these hyperplanes counted with multiplicities (resp. regardless of multiplicity).

In 1975, H. Fujimoto \cite{Fu} proved that if $q=3n+2$ and $d=1$ and without condition (i) then $f=g$. In 1983, L. Smiley \cite{S} show that there are at most two distinct meromorphic mappings sharing $3n+1$ hyperplanes of $\P^n(\C)$ regardless of multiplicity. The results of Fujimoto and Smiley are usualy called the uniqueness theorems of meromorphic mappings. 

Later on, the unicity problem for meromorphic mappings with truncated multiplicity has been extended and deepened by contribution of many authors. In \cite{ChY}, \cite{DQT}, \cite{Q11} and \cite{TQ1}, the authors have improved the above result of L. Smiley to the case where the number of hyperplanes is reduced. 
However, we may see that the condition $(*)$ is a strong restriction of such kind of results. Hence, in 2012, by introducing a new auxiliary function, Giang, Quynh and Quang \cite{GQQ} proved the following uniqueness theorem, in which the condition (*) is replaced by a more general condition (**).

\vskip0.2cm
\noindent
\textbf{Theorem A.}\ {\it Let $f,g$ be two linearly non-degenerate meromorphic mappings of $\C^m$ into $\P^n(\C)$ and let $H_1,\ldots,H_q$ be $q$ hyperplanes of $\P^n(\C)$ in general position with
$$\dim  \left (\bigcap_{j=1}^{k+1}(f,H_{i_j})^{-1}\{0\}\right) \le m-2 \quad (1 \le i_1<\cdots <i_{k+1}\le q).\ \ \ (**)$$
Assume that

(i) $(f,H_{i})^{-1}\{0\}=(g,H_{i})^{-1}\{0\}$ for all $1\le i\le q$,

(ii) $f=g$ on $\bigcup_{i=1}^q(f,H_i)^{-1}\{0\}.$

\noindent
If $q=(n+1)k+n+2$ then $f=g$.}

In 2017, using the method of \cite{GQQ}, Ru and Ugur \cite{RU} tried to improve the result of Giang-Quynh-Quang by reproving the above result with $q$ satisfying
$$ q>n+1+\frac{2knq}{q-2k+2kn}.$$
Unfortunately, in their proof, they need the following estimate (see the last inequality of the Page 10 of Lemma 4.1 in \cite{RU})
$$ q-2(n+1)k\ge\frac{q-2(n+1)k}{2nk}\sum_{j\in J}(\nu^{[n]}_{L_j(f_1)}(z)+\nu^{[n]}_{L_j(f_2)}(z)).$$
Note that this inequality only hold if $q-2(n+1)k\ge 0$. Hence all results in \cite{RU} are only valid with at least $q>2(n+1)k$ hyperplanes and this number actually is bigger than the original number $(n+1)k+n+2$ of Theorem A.

Here, we would like to note that in order to prove the uniqueness theorem, all above mentioned authors have used the following method:
 \begin{itemize}
\item Construct auxiliary functions.
\item Bound below the counting function of the auxiliary function by the sum of truncated counting functions of the pullback of hyperplanes 
\item Using Jensen's formula to bound upper the counting function of the auxiliary function by the characteristic functions of mappings.
\item Using the second main theorem (SMT) to compare the sum of truncated counting functions of the pullbacks of hyperplanes and the characteristic functions of mappings to get some contradiction if the mappings are supposed to be distinct and the number of hyperplanes is large enough.
\end{itemize}
In this paper, we find out that the last step of the above common method is not optimal. In fact, the SMT with truncated counting functions (Theorem \ref{2.2}(i)) is weaker than the SMT with counting function of the general Wronskian (Theorem \ref{2.2}(ii)). Hence, by proposing new techniques to estimate the counting function of the auxiliary function (which are different with the function used in \cite{GQQ}) and use the SMT with the counting function of the general Wronskian, we will improve Theorem A to the following.

\begin{theorem}\label{1.1}
Let $f,g$ be two linearly non-degenerate meromorphic mappings of $\C^m$ into $\P^n(\C)$ and let $H_1,\ldots.,H_q$ be $q$ hyperplanes of $\P^n(\C)$ in general position with
$$\dim  \left(\bigcap_{j=1}^{k+1}(f,H_{i_j})^{-1}\{0\}\right) \le m-2 \quad (1 \le i_1<\cdots <i_{k+1}\le q),$$
where $k$ is a positive integer, $1\le k\le n$. Assume that

(i) $(f,H_i)^{-1}\{0\}=(g,H_i)^{-1}\{0\}$ for all $1\le i\le q$,

(ii) $f=g$ on $\bigcup_{i=1}^q(f,H_i)^{-1}\{0\}.$

\noindent
If $q>(n+1)(k+1)-\frac{k^2}{2}+\frac{k}{2}+\left[\frac{k^2}{4}\right]$ then $f=g$.
\end{theorem} 
Here, the notation $[x]$ stands for the biggest integer not exceed the real number $x$.

From the above theorem, we get the following corollary.
\begin{corollary}
Let $f,g$ be two linearly non-degenerate meromorphic mappings of $\C^m$ into $\P^n(\C)$ and let $H_1,\ldots.,H_q$ be $q$ hyperplanes of $\P^n(\C)$ in general position.
Assume that

(i) $(f,H_i)^{-1}\{0\}=(g,H_i)^{-1}\{0\}$ for all $1\le i\le q$,

(ii) $f=g$ on $\bigcup_{i=1}^q(f,H_i)^{-1}\{0\}.$

\noindent
If $q>\frac{3n^2}{4}+\frac{5n}{2}$ then $f=g$.
\end{corollary}

In the last part of this paper, we will give a short survey of Nevanlinna theory and discuss such kind of this uniqueness theorem for the case of the mappings from annuli, punctured discs and balls in $\C^m$ into $\P^n(\C)$. We note that, in \cite{RU} the authors also considered the case of unit discs and punctured discs. As we mentioned above, their results are only valid if the number of involving hyperplanes is bigger than $2(n+1)k$, which is too large. 

\section{Basic notions in  Nevanlinna theory}

Throughout this paper, we use the standard notation on Nevanlinna theory from \cite{GQQ}.

\noindent
{\bf 2.1.}\ We set 
\begin{align*}
B(r) := \{ z \in \C^m : ||z|| < r\},\quad
S(r) := \{ z \in \C^m : ||z|| = r\}\ (0<r<\infty)
\end{align*}
and define 
\begin{align*}
v_m(z) :=& \big(dd^c ||z||^2\big)^{m-1}\quad \quad \text{and}\\
\sigma_m (z):=& d^c \text{log}||z||^2 \land \big(dd^c \text{log}||z||^2\big)^{m-1}
\text{on} \quad \C^m \setminus \{0\}.
\end{align*}

For a divisor $\nu$ on $\C^m$ and for a positive integer $M$ or $M= \infty$, we define the counting function of $\nu$ by
$$\nu^{[M]}(z)=\min\ \{M,\nu(z)\},$$
\begin{align*}
n^{[M]}(t,\nu) =
\begin{cases}
\int\limits_{B(t)}
\nu^{[M]}(z) v_m & \text  { if } m \geq 2,\\
\sum\limits_{|z|\leq t} \nu^{[M]} (z) & \text { if }  m=1. 
\end{cases}
\end{align*}
Define
$$ N^{[M]}(r,\nu)=\int\limits_1^r \dfrac {n^{[M]}(t,\nu)}{t^{2m-1}}dt \quad (1<r<\infty).$$
Let $\varphi : \C^m \longrightarrow \C $ be a meromorphic function.
Define
\begin{align*}
N_{\varphi}(r)=N(r,\nu_{\varphi}), \ N_{\varphi}^{[M]}(r)=N^{[M]}(r,\nu_{\varphi}).
\end{align*}
For brevity we will omit the character $^{[M]}$ if $M=\infty$.

Let $f : \C^m \longrightarrow \P^n(\C)$ be a meromorphic mapping with a reduced representation
$f = (f_0 : \dots : f_n)$. The characteristic function of $f$ is defined by 
\begin{align*}
T_f(r) = \int\limits_{S(r)} \text{log}\Vert f \Vert \sigma_m -
\int\limits_{S(1)}\text{log}\Vert f\Vert \sigma_m.
\end{align*}

\noindent

Repeating the argument in (Prop. 4.5 \cite{Fu}), we have the following:

\begin{proposition}[{see \cite[Proposition 4.5]{Fu}}]
Let $f:\C^m\rightarrow\P^n(\C)$ be a linearly non degenerate meromorphic mapping with a reduced representation $f=(f_0:\cdots:f_n).$
Then  there exist an admissible set  
$\alpha=(\alpha_0,\ldots,\alpha_n)$ with $\alpha_i=(\alpha_{i1},\ldots,\alpha_{im})\in \Z^m_+$ and $|\alpha_i|=\sum_{j=1}^{m}|\alpha_{ij}|\le i \ (0\le i \le n)$ such that the following are satisfied:

(i)\  $\{{\mathcal D}^{\alpha_i}f_0,\ldots,{\mathcal D}^{\alpha_i}f_n\}_{i=0}^{n}$ is linearly independent over the field of all meromorphic functions on $M$, i.e., \ 
$\det{({\mathcal D}^{\alpha_i}f_j)}\not\equiv 0$, (where ${\mathcal D}^{\alpha_i}f_j=\frac{\partial^{|\alpha_i|}f_j}{\partial z_1^{\alpha_{i1}}\ldots\partial z_m^{\alpha_{im}}}$). 

(ii) $\det \bigl({\mathcal D}^{\alpha_i}(h\Phi_j)\bigl)=h^{k+1}\cdot \det \bigl({\mathcal D}^{\alpha_i}\Phi_j\bigl)$ for
any nonzero meromorphic function $h$ on $\C^m.$
\end{proposition}

Let $f$ be a linearly non-degenerate meromorphic mapping and let $\alpha$ be the smallest (in the lexicographic order) admissible set w.r.t $f$. We define the general Wronskian of $f$ by
$$ W(f):=W^{\alpha}(f)= \det{({\mathcal D}^{\alpha_i}f_j)_{0\le i,j\le n}}.$$

\begin{theorem}[Second main theorem \cite{NO}]\label{2.2}
Let $f: \C^m \to \P^n(\C)$ be a linearly nondegenerate meromorphic mapping and $H_1,\ldots,H_q$ be $q$ hyperplanes in general position in $\P^n(\C).$
Then we have 
\begin{align*}
\mathrm{(i)}\ &||\ (q-n-1)T_f(r) \le \sum_{i=1}^q N_{(f,H_i)}^{[n]}(r)+o(T_f(r)),\\ 
\mathrm{(ii)}\ &||\ (q-n-1)T_f(r) \le \sum_{i=1}^q N_{(f,H_i)}(r)-N_{W(f)}(r)+o(T_f(r)). 
\end{align*}
\end{theorem}
Here, by the notation $''|| \ P''$  we mean the assertion $P$ holds for all $r \in [0,\infty)$ excluding a Borel subset $E$ of the interval $[0,\infty)$ with $\int_E dr<\infty$.

\section{Unicity theorem for meromorphic mappings on $\C^m$}
In this section, we will prove Theorem \ref{1.1}. We need the following.

Let $f$ and $g$ be as in Theorem \ref{1.1}. Assume that $f,g$ have reduced representations 
$$ f=(f_0:\cdots :f_n),\ g=(g_0:\cdots :g_n) $$ 
and hyperplanes $H_i\ (1\le i\le q)$ given by $H_i=\{a_{i0}\omega_0+\cdots+a_{in}\omega_n=0\}$. Set $h_i=\dfrac{(f,H_i)}{(g,H_i)}$.

For simplicity, we will write $\nu_{h,j}$ for $\nu_{(h,H_i)}$ with $1\le i\le q, h=f,g$. Similarly, we write $N^{[d]}_{h,i}(r)$ for $N^{[d]}_{(h,H_i)}(r)\ (h=f,g)$.

Take $\alpha=(\alpha_0,\ldots,\alpha_n)$ and $\beta=(\beta_0,\ldots,\beta_n)$ be two admissible sets w.r.t $f$ and $g$ respectively such that the general Wronskians $W(f)=W^{\alpha}(f)$ and $W(g)=W^{\beta}(f)$.

We denote by $D$ the reduced divisor whose support is defined by $\mathrm{Supp}(D)=\bigcup_{l=1}^qf^{-1}(H_l)$. 

\begin{lemma}\label{3.1} Let $f, g$ be  as in Theorem \ref{1.1}. 
Then 
$$|| \ T_g(r)=O(T_f(r))\text{ and } T_f(r)=O(T_g(r)).$$
\end{lemma}

\begin{proof}
By the Second Main Theorem, we have 
\begin{align*}
||\ (q-n-1) T_g(r)\le &\sum_{i=1}^q N_{(g,H_i)}^{[n]}(r)+o(T_g(r))\\
\le &\sum_{i=1}^qnN_{(g,H_i)}^{[1]}(r)+o(T_g(r))=\sum_{i=1}^qnN_{(f,H_i)}^{[1]}(r)+o(T_g(r))\\
\le & qn\ T_f(r)+o(T_f(r)+T_g(r)).
\end{align*}
Thus
\begin{align*}
||(q-n-1)\ T_g(r)\le qn T_f(r)++o(T_f(r)+T_g(r)).
\end{align*}
Hence \quad $|| \quad T_g(r)=O(T_f(r)).$ Similarly, we get \  \ $|| \ \ T_f(r)=O(T_g(r)).$
\end{proof}

\begin{lemma}\label{3.2} 
Let $f,g$ be as in the Theorem \ref{1.1}. Suppose that there exist two indices $1\le i<j\le q$ such that $h_i\ne h_j$. Then we have
\begin{align*}
||\ (q-n-1)&\left (N(r,\min\{\nu_{f,i},\nu_{g,i}\})+N(r,\min\{\nu_{f,j},\nu_{g,j}\})-N^{[1]}(r,\nu_{f,i}+\nu_{f,j})+N(r,D)\right)\\
&\le\sum_{h=f,g}(\sum_{l=1}^qN_{h,l}(r)-N_{W(h)}(r)+o(T_h(r))).
\end{align*}
\end{lemma}
\begin{proof}
Consider the following nonzero holomorphic function
$$ P=(f,a_i)(g,a_j)-(f,a_j)(g,a_i).$$
We see that
\begin{itemize}
\item If $z\in\mathrm{Supp}(D)\setminus\{(f,a_i)(f,a_j)=0\}$ then 
\begin{align*}
\nu_P(z)\ge 1= &\min\{\nu_{f,i}(z),\nu_{g,i}(z)\}+\min\{\nu_{f,j}(z),\nu_{g,j}(z)\}\\
&-\min\{1,\nu_{f,i}(z)+\nu_{f,j}(z)\}+D(z).
\end{align*}
\item If $z\not\in\mathrm{Supp}(D)\setminus\{(f,a_i)(f,a_j)=0\}$ then
\begin{align*}
\nu_P(z)&\ge\min\{\nu_{f,i}(z),\nu_{g,i}(z)\}+\min\{\nu_{f,j}(z),\nu_{g,j}(z)\}\\
&=\min\{\nu_{f,i}(z),\nu_{g,i}(z)\}+\min\{\nu_{f,j}(z),\nu_{g,j}(z)\}\\
&\ \ \ -\min\{1,\nu_{f,i}(z)+\nu_{f,j}(z)\}+D(z).
\end{align*}
\end{itemize}
Therefore, we always have
$$ \nu_P(z) \ge \min\{\nu_{f,i}(z),\nu_{g,i}(z)\}+\min\{\nu_{f,j}(z),\nu_{g,j}(z)\}-\min\{1,\nu_{f,i}(z)+\nu_{f,j}(z)\}+D(z).$$
Integrating both sides of this inequality, we get
\begin{align}\label{new1}
\begin{split}
N_P(r)\ge &N(r,\min\{\nu_{f,i},\nu_{g,i}\})+N(r,\min\{\nu_{f,j},\nu_{g,j}\})\\
&-N^{[1]}(r,\nu_{f,i}+\nu_{f,j})+N(r,D).
\end{split}
\end{align}

On the other hand, by Jensen's formula and the second main theorem we have

\begin{align}\label{new2}
\begin{split}
||\ N_P(r)&=\int\limits_{S(r)}\log|P(z)|\sigma_m+O(1)\\
&\le\sum_{h=f,g}\int\limits_{S(r)}\log (|(h,a_i)|^2+|(h,a_j)|^2)^{1/2}\sigma_m+O(1)\\
&=\sum_{h=f,g}\left (\int\limits_{S(r)}\log ||h||\sigma_m+o(T_h(r))\right )\\
&=T_f(r)+T_g(r)+o(T_f(r)+T_g(r))\\
&\le\frac{1}{q-n-1}\sum_{h=f,g}(\sum_{l=1}^qN_{h,l}(r)-N_{W(h)}(r)+o(T_h(r))).
\end{split}
\end{align}
From (\ref{new1}) and (\ref{new2}), we get the desired conclusion.
\end{proof}

\begin{lemma}\label{3.5}
Let $f,g$ be as in Theorem \ref{1.1}, $W^{\alpha}(f)$ and $W^{\beta}(g)$ be the general Wronskians of $f$ and $g$ respectively, where $\alpha=(\alpha_0,\ldots,\alpha_n),\beta=(\beta_0,\ldots,\beta_n)$, $|\alpha_i|\le i,|\beta_i|\le i\ (0\le i\le n)$. Let $\lambda=\dfrac{q}{nk-\frac{k^2}{2}+\frac{3k}{2}+\left [\frac{k^2}{4}\right]}.$ We have
\begin{align*}
\lambda\biggl (\sum_{l=1}^q&(N_{f,l}(r)+N_{g,l})-N_{W^{\alpha}(f)}(r)-N_{W^{\beta}(g)}(r)\biggl ) \le \sum_{l=1}^q\biggl (N(r,\min\{\nu_{f,l},\nu_{g,l}\})&\\
&+(N(r,\min\{\nu_{f,\sigma(l)},\nu_{g,\sigma(l)}\})-N^{[1]}(r,\nu_{f,l}+\nu_{f,\sigma(l)})+N(r,D)\biggl)+o(T_f(r)).
\end{align*}
\end{lemma}

\begin{proof}
It is enough for us to show that
\begin{align}\label{new3}
\begin{split}
\lambda&\biggl (\sum_{l=1}^q(\nu_{f,l}+\nu_{g,l})-\nu_{W^\alpha(f)}-\nu_{W^\beta(g)}\biggl)\\
&\le\sum_{l=1}^q\bigl (\min\{\nu_{f,l},\nu_{g,l}\}+\min\{\nu_{f,\sigma(l)},\nu_{g,\sigma(l)}\}-\min\{1,\nu_{f,l}+\nu_{f,\sigma(l)}\}+D\bigl).
\end{split}
\end{align}
For simplicity, we denote by $P$ and $Q$ the right hand side and the left hand side of (\ref{new3}) respectively.
Fix a point $z\in\bigcup_{l=1}^q(f,H_l)^{-1}\{0\}\setminus I(f)$. We denote by $J$ the set of all indices $j\in\{1,\ldots,q\}$ such that $(f,H_j)(z)=0$. Then we see that $t=\sharp J\le k.$ We set $J=\{j_0,\ldots,j_{t-1}\}$ with $1\le j_0<\cdots<j_{t-1}\le q$. Then we have 
\begin{align*}
Q(z)&\ge 2\sum_{l=0}^{t-1}\min\{\nu_{f,j_l}(z),\nu_{g,j_l}(z)\}-\sum_{l=1}^q\min\{1,\nu_{f,l}(z)+\nu_{f,\sigma(l)}(z)\}+q\\
&\ge 2\sum_{l=0}^{t-1}\min\{\nu_{f,j_l}(z),\nu_{g,j_l}(z)\}-2t+q.
\end{align*}
Let $u,v$ be two arbitrary permutations of $\{0,\ldots,n\}$. We note that, for four positive integers $a,b,c,d$, we have $\min\{a,b\}\ge\min\{a,c\}+\min\{b,d\}-\max\{d,c\}$. Then
\begin{align*}
\min\{\nu_{f,j_l}(z),\nu_{g,j_l}(z)\}&\ge\min\{\nu_{f,j_l},|\alpha_{u(l)}|\}+\{\nu_{g,j_l},|\beta_{v(l)}|\}-\max\{|\alpha_{u(l)}|,|\beta_{v(l)}|\}\\
&\ge\min\{\nu_{f,j_l},|\alpha_{u(l)}|\}+\{\nu_{g,j_l},|\beta_{v(l)}|\}-\max\{u(l),v(l)\},
\end{align*} 
Now, it is easy to see that
\begin{align*}
\sum_{l=0}^{t-1}\max\{u(l),v(l)\}&\le\max_{0\le s\le t}[\sum_{l=n-s+1}^nl+\sum_{l=n-t+s+1}^nl]\\
&=\max_{0\le s\le t}\left (\frac{s(2n-s+1)}{2}+\frac{(t-s)(2n-t+s+1)}{2}\right )\\
&=\max_{0\le s\le t}\frac{2nt-t^2+2ks-2s^2+t}{2}=nt-\frac{1}{2}(t^2-t)+\left [\frac{t^2}{4}\right ]\\
\end{align*}
This implies that 
\begin{align}\label{new4}
\begin{split}
\sum_{l=0}^{t-1}\min\{\nu_{f,j_l}(z),\nu_{g,j_l}(z)\}\ge&\sum_{l=0}^{t-1}(\min\{\nu_{f,j_l},|\alpha_{u(l)}|\}+\{\nu_{g,j_l},|\beta_{v(l)}|\})\\
&-nt+\frac{1}{2}(t^2-t)-\left [\frac{t^2}{4}\right ].
\end{split}
\end{align}
Then from (\ref{new4}), we have
\begin{align}\label{new5}
\begin{split}
Q(z)\ge &(2-\lambda)\sum_{l=0}^{t-1}\min\{\nu_{f,j_l}(z),\nu_{g,j_l}(z)\}+\lambda\sum_{l=0}^{t-1}\min\{\nu_{f,j_l}(z),\nu_{g,j_l}(z)\}-2t+q\\
\ge &\lambda\biggl (\sum_{l=0}^{t-1}(\min\{\nu_{f,j_l},|\alpha_{u(l)}|\}+\{\nu_{g,j_l},|\beta_{v(l)}|\})-nt+\frac{1}{2}(t^2-t)-\left [\frac{t^2}{4}\right ]\biggl )\\
&+(2-\lambda)t-2t+q\\
=&\lambda(\sum_{l=0}^{t-1}(\min\{\nu_{f,j_l},|\alpha_{u(l)}|\}+\{\nu_{g,j_l},|\beta_{v(l)}|\})+q-(nt-\frac{t^2}{2}+\frac{3t}{2}+\left [\frac{t^2}{4}\right])\lambda\\
\ge&\lambda(\sum_{l=0}^{t-1}(\min\{\nu_{f,j_l},|\alpha_{u(l)}|\}+\{\nu_{g,j_l},|\beta_{v(l)}|\}).
\end{split}
\end{align}

On the other hand, by the usual argument in Nevanlinna theory, we have
\begin{align*}
\sum_{l=1}^q\nu_{f,l}(z)-\nu_{W^\alpha(f)}(z)&=\sum_{l=0}^{t-1}\nu_{f,i_l}(z)-\nu_{W^\alpha(f)}(z)\\
&\le\sum_{l=0}^{t-1}\nu_{f,i_l}(z)-\min_{u}\sum_{l=0}^{t-1}\nu_{\mathcal D^{\alpha_{u(l)}}(f,a_{i_l})(z)}\\
&\le\sum_{l=0}^{t-1}\nu_{f,i_l}(z)-\min_{u}\sum_{l=0}^{t-1}\max\{0;\nu_{(f,a_{i_l})}(z)-|\alpha_{u(l)}|\}\\
&\le\max_{u}\sum_{l=0}^{t-1}\min\{\nu_{f,i_l}(z),|\alpha_{u(l)}|\},
\end{align*}
where $\min_u$ and $\max_u$ is taken over all permutations $u$ of $\{0,\ldots,n\}$. Similarly
$$ \sum_{l=1}^q\nu_{g,l}(z)-\nu_{W^\beta(g)}(z)\le\max_{v}\sum_{l=0}^{t-1}\min\{\nu_{g,i_l}(z),|\alpha_{v(l)}|\},$$
where $\max_v$ is taken over all permutations $v$ of $\{0,\ldots,n\}$. Therefore we have
\begin{align}\label{new6}
 P(z)\le\lambda\max_{u,v}\sum_{l=0}^{t-1}\left (\min\{\nu_{f,i_l}(z),|\alpha_{u(l)}|\}+\min\{\nu_{g,i_l}(z),|\alpha_{v(l)}|\}\right).
\end{align}
Hence, from (\ref{new5}) and (\ref{new6}), we obtain
$$ P(z)\le Q(z). $$
This proves the desired inequality of the lemma.
\end{proof}

\begin{proof}[Proof of Theorem \ref{1.1}]
Suppose that $f\ne g$. By changing indices if necessary, we may assume that
$$\underbrace{\dfrac{(f,H_1)}{(g,H_1)}\equiv \dfrac{(f,H_2)}{(g,H_2)}\equiv \cdot\cdot\cdot\equiv \dfrac{(f,H_{k_1})}
{(g,H_{k_1})}}_{\text { group } 1}\not\equiv
\underbrace{\dfrac{(f,H_{k_1+1})}{(g,H_{k_1+1})}\equiv \cdot\cdot\cdot\equiv \dfrac{(f,H_{k_2})}{(g,H_{k_2})}}_{\text { group } 2}$$
$$\not\equiv \underbrace{\dfrac{(f,H_{k_2+1})}{(g,H_{k_2+1})}\equiv \cdot\cdot\cdot\equiv \dfrac{(f,H_{k_3})}{(g,H_{k_3})}}_{\text { group } 3}\not\equiv \cdot\cdot\cdot\not\equiv \underbrace{\dfrac{(f,H_{k_{s-1}+1})}{(g,H_{k_{s-1}+1})}\equiv\cdot\cdot\cdot \equiv 
\dfrac{(f,H_{k_s})}{(g,H_{k_s})}}_{\text { group } s},$$
where $k_s=q.$ 

For each $1\le i \le q,$ we set
\begin{equation*}
\sigma (i)=
\begin{cases}
i+n& \text{ if $i+n\leq q$},\\
i+n-q&\text{ if  $i+n> q$},
\end{cases}
\end{equation*}
and  
$$P_i=(f,H_i)(g,H_{\sigma (i)})-(g,H_i)(f,H_{\sigma (i)}).$$
Since  $f\not\equiv g,$ the number of elements of each group is at most $n$. Then $\dfrac{(f,H_i)}{(g,H_i)}$ and 
$\dfrac{(f,H_{\sigma (i)})}{(g,H_{\sigma (i)})}$ belong to distinct groups. Therefore $P_i\not\equiv 0\ (1\le i\le q)$. 

Then by Lemma \ref{3.1}, we have
\begin{align}\label{new7}
\begin{split}
||\ \sum_{i=1}^q&\biggl (N(r,\min\{\nu_{f,i},\nu_{g,i}\})+N(r,\min\{\nu_{f,\sigma(i)},\nu_{g,\sigma(i)}\})-N^{[1]}(r,\nu_{f,i}+\nu_{f,\sigma(i)})\\
&+N(r,D)\biggl )\le \frac{q}{q-n-1}\sum_{h=f,g}(\sum_{l=1}^qN_{h,l}(r)-N_{W(h)}(r)+o(T_h(r))),
\end{split}
\end{align}
where $W(f)=W^{\alpha}(f), W(g)=W^{\beta}(g)$ are the general Wronskians of $f$ and $g$ respectively.

On the other hand, by Lemma \ref{3.2} we have
\begin{align}\label{new8}
\begin{split}
\sum_{h=f,g}(\sum_{l=1}^qN_{h,l}(r)&-N_{W(h)}(r))\le\frac{1}{\lambda}\sum_{i=1}^q\biggl (N(r,\min\{\nu_{f,i},\nu_{g,i}\})\\
&+N(r,\min\{\nu_{f,\sigma(i)},\nu_{g,\sigma(i)}\})-N^{[1]}(r,\nu_{f,i}+\nu_{f,\sigma(i)})+N(r,D)\biggl)
\end{split}
\end{align}

Combining (\ref{new7}) and (\ref{new8}), we get
$$||\ \sum_{h=f,g}(\sum_{l=1}^qN_{h,l}(r)-N_{W(h)}(r))\le\dfrac{q}{\lambda(q-n-1)}(\sum_{l=1}^qN_{h,l}(r)-N_{W(h)}(r))+o(T_h(r)),$$
$$ \text{i.e.,}||\ \left (1-\dfrac{nk-\frac{k^2}{2}+\frac{3k}{2}+\left[\frac{k^2}{4}\right]}{q-n-1}\right )\sum_{h=f,g}(\sum_{l=1}^qN_{h,l}(r)-N_{W(h)}(r))\le o(T_h(r)).$$
We note that 
$$1-\frac{nk-\frac{k^2}{2}+\frac{3k}{2}+\left[\frac{k^2}{4}\right]}{q-n-1}=\frac{q-(n+1)(k+1)+\frac{k^2}{2}-\frac{k}{2}-\left[\frac{k^2}{4}\right]}{q-n-1}>0$$
and (by the second main theorem)
$$||\ \sum_{l=1}^qN_{h,l}(r)-N_{W(h)}(r)\ge (q-n-1)T_h(r)+o(T_h(r)). $$
Hence, the above inequality implies that
$$ \biggl (q-(n+1)(k+1)+\frac{k^2}{2}-\frac{k}{2}-\left[\frac{k^2}{4}\right]\biggl)(T_f(r)+T_g(r))=o(T_f(r)).$$
This is a contradiction. Then the supposition is impossible.

Therefore, we have $f=g$.
\end{proof}

\section{Unicity problem for holomorphic curves of annuli, punctured disks and meromorphic mappings of balls into $\P^n(\C)$}

\noindent
(a) Holomorphic curves on annuli.

For $R_0>1$, denote by $\A (R_0)$ the annulus defined by 
$$\A (R_0)=\{z;\frac{1}{R_0},|z|<R_0\}.$$
For a divisor $\nu$ on $\A (R_0)$, which we may regard as a function on  $\A (R_0)$ with values in $\mathbb Z$ whose support is a discrete subset of $\A (R_0),$ and  for a positive integer $M$ (maybe $M= \infty$), we define the counting function of $\nu$ as follows
\begin{align*}
n_0^{[M]}(t)&=\begin{cases}
\sum\limits_{1\le |z|\le t}\min\{M,\nu (z)\}&\text{ if }1\le t<R_0\\
\sum\limits_{t\le |z|<1}\min\{M,\nu (z)\}&\text{ if }\dfrac{1}{R_0}<t< 1
\end{cases}\\
 \text{ and }N_0^{[M]}(r,\nu)&=\int\limits_{\frac{1}{r}}^1 \dfrac {n_0^{[M]}(t)}{t}dt +\int\limits_1^r \dfrac {n_0^{[M]}(t)}{t}dt \quad (1<r<\infty).
\end{align*}
For brevity we will omit the character $^{[M]}$ if $M=\infty$.

Let $f$ be a holomorphic mapping from an annulus $\A (R_0)$ into $\P^n(\C)$ with a reduced representation $f=(f_0:\cdots :f_n)$. The characteristic function of $f$ is defined by
\begin{align}\label{newdef1}
 T_0(r,f)=\dfrac{1}{2\pi}\int\limits_{0}^{2\pi}\log ||f(re^{i\theta})|| d\theta +\dfrac{1}{2\pi}\int\limits_{0}^{2\pi}\log ||f(\frac{1}{r}e^{i\theta})|| d\theta-\dfrac{1}{\pi}\int\limits_{0}^{2\pi}\log ||f(e^{i\theta})|| d\theta .
\end{align}

A subset $\Delta$ of $[1;R_0)$ is said to be an $\Delta_{R_0}$-set if it satisfies $\int_{\Delta}\dfrac{dr}{(R_0-r)^{\lambda +1}} <+\infty$ for some $\lambda \ge 0$.
We denote by $S_f(r)$ quantities satisfying:
$$ S_f(r)=O\left (\log\biggl (\dfrac{T_0(r,f)}{R_0-r}\biggl )\right )\text{ as }r\longrightarrow R_0$$ 
for $r\in (1,R_0)$ except for an $\Delta_{R_0}$-set.
The holomorphic curve $f$ above is said to be admissible if it satisfies
$$\underset{r\longrightarrow R_0^-}{\mathrm{limsup}}\dfrac{T_0(r,f)}{-\log (R_0-r)}=+\infty.$$
Thus for an admissible holomorphic curve $f$, we have $S_f(r) = o(T_0(r, f ))$ as $r\longrightarrow R_0$ for all $1\le r<R_0$ except for an $\Delta_{R_0}$-set mentioned above.

Using the same argument as in the proof of the SMT for holomorphic curves from $\C$ into $\P^n(\C)$, the authors in \cite{QGH} proved the following SMT for holomorphic curves from an annulus into $\P^n(\C)$.

\begin{theorem}[{see \cite[Theorem 3.2]{QGH}}]\label{4.1}
Let $f:\mathbb{A}(R_0)\to\mathbb{P}^n(\mathbb{C})$ be a linearly nondegenerate holomorphic mapping with the Wronskian $W(f)$. Let $\{H_i\}_{i=1}^q$ $(q\ge n+2)$ be a set of $q$ hyperplanes in $\mathbb{P}^N(\mathbb{C})$ in general position. Then
\begin{align*}
\mathrm{(i)}&\ (q-N-1) T_0 (r,f)\le \sum_{i=1}^{q}N_0^{[n]}(r,\nu_{(f,H_i)})+S_f(r).\\ 
\mathrm{(ii)}&\ (q-N-1) T_0 (r,f)\le \sum_{i=1}^{q}N_0(r,\nu_{(f,H_i)})-N_0(r,\nu_{W(f)})+S_f(r). 
\end{align*} 
\end{theorem}
Here, we note that in \cite[Theorem 3.2]{QGH} there is only the assertion (i), but the assertion (ii) appears in the proof.

Hence, the second main theorem in this case is completely same as the case of holomorphic curves on $\C$. Then by repeating the same lines of the proof of Theorem \ref{1.1} we easily get the following theorem.

\begin{theorem}\label{4.2}
Let $f,g$ be two linearly nondegenerate admissible holomorphic curves of $\A(R_0)$ into $\P^n(\C)$ and let $H_1,\ldots.,H_q$ be $q$ hyperplanes of $\P^n(\C)$ in general position with
$$\dim  \left(\bigcap_{j=1}^{k+1}f^{-1}(H_{i_j})\right) =\emptyset \quad (1 \le i_1<\cdots <i_{k+1}\le q).$$
Assume that

(i) $f^{-1}(H_i)=g^{-1}(H_i)$ for all $1\le i\le q$,

(ii) $f=g$ on $\bigcup_{i=1}^qf^{-1}(H_i).$

\noindent
If $q>(n+1)(k+1)-\frac{k^2}{2}+\frac{k}{2}+\left[\frac{k^2}{4}\right]$ then $f=g$.
\end{theorem} 

\noindent
(b) Holomorphic curve on a punctured disc.

We set punctured discs on $\hat{\C}=\C\cup\{\infty\}$ about $\infty$ by
\begin{align*}
 \Delta^* &=\{z\in\C\ :\ |z|\ge 1\},\\
 \Delta^*(t) &=\{z\in\C \ :\ |z|\ge t\}, \quad t\ge 1.
\end{align*}
Let $E$ be a divisor on $\Delta^*$. The {\it truncated counting function to level $d$} of $E$ defined by
\begin{align*}
N^{[d]}(r,E) := \int\limits_1^r \frac{n^{[d]}(t,E)}{t}dt\quad
(1 < r < +\infty).
\end{align*}
We simply write  $N(r,E)$ for $N^{[+\infty]}(r,E).$

Let $f:\Delta^*\rightarrow\P^n(\C )$ be a holomorphic curve. There exist a neighborhood $U$ of $\Delta^*$ in $\C^m$ and a reduced representation $(f_0 : \cdots : f_n)$ on $U$ of $f$.
We set
$||f||:=(|f_0|^2+\cdots +|f_n|^2)^{\frac{1}{2}}$.

Denote by $\Omega$ the Fubini - Study form of $\P^n(\C )$. The characteristic function of $f$ with respect to $\Omega$ is defined by
\begin{align*}
T_f(r)=\dfrac{1}{2\pi}\int_{S (r)}\log ||f(re^{i\theta})||d\theta -
 \dfrac{1}{2\pi}\int_{S (1)}\log ||f(e^{i\theta})||d\theta
\end{align*}

For a meromorphic function $\varphi$ on $\Delta^*$, the proximity  function $m(r,\varphi)$ is defined by
$$m(r,\varphi) := \frac{1}{2\pi}\int\limits_{S(r)} \log^+ |\varphi| d\theta .$$
Due to Noguchi, we have the following lemma on logarithmic derivative as follows.
\begin{theorem}[{Lemma on logarithmic derivative \cite{Nog81}}]
Let $\varphi$ be a nonzero meromorphic function on $\Delta^*.$ Then
\begin{align*}
\biggl|\biggl|\quad m\biggl(r,\dfrac{\varphi '}{\varphi}\biggl)
= O (\log^+T_\varphi (r))+ O(\log r).
\end{align*}
\end{theorem}

Thank to this theorem and repeating the argument as in the proof of the second main theorem in $\C^m$, we have the following theorem.

\begin{theorem}[{Second main theorem}]\label{4.4}
Let $f:\Delta^*\to\mathbb{P}^n(\mathbb{C})$ be a linearly nondegenerate holomorphic curve with the Wronskian $W(f)$. Let $\{H_i\}_{i=1}^q$ $(q\ge n+2)$ be a set of $q$ hyperplanes in $\mathbb{P}^N(\mathbb{C})$ in general position. Then
\begin{align*}
\mathrm{(i)}&\ ||\ (q-n-1) T_f(r)\le \sum_{i=1}^{q}N_0^{[n]}(r,\nu_{(f,H_i)})+o(T_f(r)).\\ 
\mathrm{(ii)}&\ ||\ (q-n-1) T_f(r)\le \sum_{i=1}^{q}N_0(r,\nu_{(f,H_i)})-N_0(r,\nu_{W(f)})+o(T_f(r)). 
\end{align*} 
\end{theorem}

Then, repeating the same lines of the proof of Theorem \ref{1.1} we get a similar theorem as follows.

\begin{theorem}\label{4.2}
Let $f,g$ be two linearly nondegenerate admissible holomorphic curves of $\Delta^*$ into $\P^n(\C)$ and let $H_1,\ldots.,H_q$ be $q$ hyperplanes of $\P^n(\C)$ in general position with
$$\dim  \left(\bigcap_{j=1}^{k+1}f^{-1}(H_{i_j})\right) =\emptyset \quad (1 \le i_1<\cdots <i_{k+1}\le q).$$
Assume that

(i) $f^{-1}(H_i)=g^{-1}(H_i)$ for all $1\le i\le q$,

(ii) $f=g$ on $\bigcup_{i=1}^qf^{-1}(H_i).$

\noindent
If $q>(n+1)(k+1)-\frac{k^2}{2}+\frac{k}{2}+\left[\frac{k^2}{4}\right]$ then $f=g$.
\end{theorem}

\noindent
(c) Meromorphic mappings from balls.

Let $R_0>0$. Set $B(R_0)=\{z\in\C^m\ :\ ||z||<R_0\}$ and fix a positive number $r_0<R_0$.

Let $\phi$ be a non-zero meromorphic function on $B(R)$. The truncated (to level $d$) counting function of zero of $\phi$ is defined by
$$ N^{[d]}_{\phi}(r,r_0)=\int\limits_{r_0}^r\frac{n^{[d]}(t,\nu^0_{\phi})}{t}dt. $$
Let $f$ be a holomorphic mapping of $B(R_0)$ into $\P^n(\C)$ with a reduce representation $f = (f_0:\ldots :f_n)$ on $B(R_0)$ (since $B(R_0)$ is Cauchy II domain).
The characteristic function is define by
$$T_f(r,r_0) :=\int\limits_{S(r)}\log ||f||\sigma_m- \int\limits_{S(r_0)}\log ||f||\sigma_m.$$
The meromorphic mapping $f$ is said to be admissible if
$$ \underset{r\rightarrow R^{-}}{\lim\mathrm{sup}}\frac{T_f(r,r_0)}{-\log (R-r_0)}=+\infty.$$
If $f$ is linearly non-degenerate then the general Wronskian of $f$ is defined as the same as in the case of meromorphic mappings on $\C^m$.
Then we have the following theorem due to Fujimoto \cite{Fu86}.
\begin{theorem}[{see \cite[Theorem 2.13]{Fu86}}]\label{4.3}
Let $f:B(R_0)\to\mathbb{P}^n(\mathbb{C})$ be a linearly nondegenerate holomorphic mapping with the general Wronskian $W(f)$. Let $\{H_i\}_{i=1}^q$ $(q\ge n+2)$ be a set of $q$ hyperplanes in $\mathbb{P}^n(\mathbb{C})$ in general position. Then
\begin{align*}
\mathrm{(i)}&\ (q-n-1) T_f (r,r_0)\le \sum_{i=1}^{q}N^{[n]}_{(f,H_i)}(r)+O(\log^+T_f(r,r_0+\log\frac{1}{R_0-r}))\\ 
\mathrm{(ii)}&\ (q-n-1) T_f (r,r_0)\le \sum_{i=1}^{q}N_{(f,H_i)}(r)-N_{W(f)}(r)+O(\log^+T_f(r,r_0+\log\frac{1}{R_0-r})) 
\end{align*} 
for all $r\in [r_0;R_0)$ outside a Borel subset $E$ of $[r_0;R_0)$ with $\int\limits_E\frac{dt}{R_0-t}\le +\infty.$
\end{theorem}

Therefore, the second main theorem in this case is totally similar to that in the case of $\C^m$. Hence, by repeating the same lines of the proof of Theorem \ref{1.1} we also get a similar theorem.

\begin{theorem}\label{4.2}
Let $f,g$ be two linearly nondegenerate admissible meromorphc mappings of $B(R_0)$ into $\P^n(\C)$ and let $H_1,\ldots.,H_q$ be $q$ hyperplanes of $\P^n(\C)$ in general position with
$$\dim  \left(\bigcap_{j=1}^{k+1}(f,H_{i_j})^{-1}\{0\}\right) \le m-2 \quad (1 \le i_1<\cdots <i_{k+1}\le q).$$
Assume that

(i) $(f,H_i)^{-1}\{0\}=(g,H_i)^{-1}\{0\}$ for all $1\le i\le q$,

(ii) $f=g$ on $\bigcup_{i=1}^q(f,H_i)^{-1}\{0\}.$

\noindent
If $q>(n+1)(k+1)-\frac{k^2}{2}+\frac{k}{2}+\left[\frac{k^2}{4}\right]$ then $f=g$.
\end{theorem} 

\subsection*{Acknowledgements}
This research is funded by Vietnam National Foundation for Science and Technology Development (NAFOSTED) under grant number 101.04-2018.01.

\vskip0.2cm
{\footnotesize 
\noindent
{\sc Si Duc Quang}\\
$^1$Department of Mathematics, Hanoi National University of Education, 136-Xuan Thuy - Cau Giay - Hanoi, Vietnam,\\
$^2$Thang Long Instutute of Mathematics and Applied Sciences, Nghiem Xuan Yem, Hoang Mai, Hanoi,\\
\textit{E-mail}: quangsd@hnue.edu.vn

\end{document}